\documentclass[12pt,a4]{amsart}
\usepackage[utf8]{inputenc}
\usepackage{a4wide}
\usepackage[T1]{fontenc}
\usepackage{fixltx2e, graphicx, longtable, float, wrapfig, soul,
  textcomp, marvosym, wasysym, latexsym,  bbm,
  fancybox,fancyvrb, cite,amsmath,amssymb,amsthm,amsrefs,color,subfig,pdfpages,enumerate}
\usepackage{enumitem}
\usepackage{listings}
\usepackage{tabularx}
\usepackage{pdfpages}
\usepackage{mathtools}
 \usepackage{relsize}
\usepackage{blkarray, bigstrut}
\usepackage[all]{xy}
\newtheorem{theorem}{Theorem}[section]
\newtheorem{lemma}[theorem]{Lemma}
\newtheorem{proposition}[theorem]{Proposition}
\newtheorem{corollary}[theorem]{Corollary}

\theoremstyle{definition}
\newtheorem{example}[theorem]{Example}
\newtheorem{definition}[theorem]{Definition}
\newtheorem{remark}[theorem]{Remark}

\numberwithin{equation}{section}

\title{Homogeneous completely simple semigroups}
\author{Thomas Quinn-Gregson}

\email{thomas.quinn-gregson@tu-dresden.de}

\address{Institut f{\"u}r Algebra\\ TU Dresden \\ Dresden}

\date{\today}



\thanks{This research  has been funded by a Postdoctoral Fellowship from the Department of Mathematics of the University of York and by the European Research Council (Grant Agreement No. 681988, CSP-Infinity).}

\begin{document}

\begin{abstract}
A semigroup is \textit{completely simple} if it has no proper ideals and contains a primitive idempotent. We say that a completely simple semigroup $S$ is a \textit{homogeneous completely simple semigroup} if any isomorphism between finitely generated completely simple subsemigroups of $S$ extends to an automorphism of $S$. Motivated by the study of homogeneous completely regular semigroups, we obtain a complete classification of homogeneous completely simple semigroups, modulo the group case. As a consequence, all  finite regular homogeneous semigroups are described, thus extending the work of Cherlin on homogeneous finite groups. 
\end{abstract}

\maketitle

\section{Introduction}

 A countable first order structure $M$ is \textit{homogeneous} if any isomorphism between finitely generated  substructures extends to an automorphism of $M$. 
Interest in homogeneity stems from the strong connections between  homogeneity and model theoretic concepts including $\aleph_0$-categoricity and quantifier elimination \cite{Hodges97}.  In particular, a homogeneous structure which is \textit{uniformly locally finite} (ULF) and has finite signature is $\aleph_0$-categorical and has quantifier elimination (where a structure is ULF if there exists a function $f:\mathbb{N}\rightarrow \mathbb{N}$ such that, for any $n\in \mathbb{N}$, each $n$-generated substructure has at most $f(n)$ elements). 

Progress has been made in classifiying homogeneous groups and rings (see, for example, \cite{Cherlin93}, \cite{Saracino84}), and   has been completed for finite groups in \cite{Cherlin2000}, and solvable groups in \cite{Cherlin91}, up to the determination of the homogeneous nilpotent groups of class 2 and exponent 4. Homogeneous semilattices were determined in \cite{Truss99}, and this work was considerably extended to both homogeneous idempotent semigroups (bands)  and inverse semigroups  by the author in \cite{Quinnband} and \cite{Quinninv}, respectively. While a  classification of homogeneous bands was achieved, a number of open problems still exist for the inverse case.

A semigroup is \textit{completely regular} if every element is contained in a subgroup. Completely regular semigroups were first studied by Clifford  \cite{Cliffordcr}, although for an in depth study we refer the reader to Petrich and Reilly's monograph \cite{Petrich99}. Clifford called these semigroups `semigroups admitting relative inverses', since every element possesses a unique inverse in the maximal subgroup in which it lies. As a consequence, the class of completely regular semigroups forms a variety $\mathcal{CRS}$ of unary semigroups, that is, semigroup equipped with an additional, basic, unary operation (in this case the operation mapping an element to its relative inverse). We may thus define a completely regular semigroup to be \textit{homogeneous} if it is homogeneous as a unary semigroup. In this setting a `substructure' is a completely regular subsemigroup, and therefore this is a natural  choice of signature. Indeed, if we consider  homogeneity of a completely regular semigroup in the setting of semigroups, substructures need not be completely regular: this subtle variant is considered in the final section. Unless stated otherwise we will consider completely regular semigroups as unary semigroups. 

Bands (i.e. idempotent semigroups) form an important subvariety of $\mathcal{CRS}$, as do completely simple semigroups, which we now define. A semigroup without zero is called \textit{simple} if it has no proper ideals. A simple semigroup is \textit{completely simple} if it contains an idempotent which is minimal within the set of idempotents $E(S)$ of $S$ under the natural order. That is, if it contains an idempotent $e$ such that 
\[ (\forall f\in E(S)) \quad ef=fe=f \Rightarrow f=e. 
\] 

Clifford \cite{Cliffordcr} showed that every completely regular semigroup can be written as a semilattice $Y$ of completely simple semigroups $S_{\alpha}$ ($\alpha \in Y$). Further details of completely regular semigroups and this decomposition can be found in \cite{Petrich99}. A simple generalization of the band case in \cite{Quinnband} yields the following motivating result: 

\begin{theorem} Let $S=\bigcup_{\alpha\in Y} S_{\alpha}$ be a completely regular semigroup. If $S$ is homogeneous then each completely simple semigroup $S_{\alpha}$ is homogeneous, and  the $S_{\alpha}$'s are pairwise isomorphic.
\end{theorem} 

 As a consequence, we first require a complete understanding of homogeneous completely simple semigroups before we begin to tackle the general case. This is our central aim: a complete classification of homogeneous completely simple semigroups.

We note that investigations into the model theoretic properties of completely simple semigroups were initiated by the author in \cite{Quinncat2}, for the case of $\aleph_0$-categorical completely (0-)simple semigroups. Further (and more complex) examples can be obtained from this paper by considering those homogeneous completely simple semigroup which are ULF. 

This paper proceeds as follows. In Section 2, the method of Fra\"iss\'e to determine homogeneity is transferred into the setting of completely simple semigroups. In Section 3, we review the work on the homogeneity of edge-coloured bipartite graphs by Jenkinson, Seidel and Truss. From the Rees' Theorem, every completely simple semigroup is isomorphic to a \textit{Rees matrix semigroup}, where the latter is determined by a group $G$ and a \textit{sandwich matrix} over $G$. From this, every completely simple semigroup is shown in Section 4 to induce an edge-coloured bipartite graph.  
We end Section 4 by deriving a number of consequences of the isomorphism theorem for Rees matrix semigroups, which are used throughout Section 5 to understand the role of the underlying group and sandwich matrix of a homogeneous Rees matrix semigroup.  In particular, we show that the homogeneity of a completely simple semigroup depends only on that of its maximal subgroups and its idempotent generated subsemigroup, and a complete classification is obtained. 
Finally, in Section 6 we consider the stronger notion of homogeneity of a completely simple semigroup  as a semigroup, which is motivated by the homogeneity of regular semigroups with either a finite number of idempotents, or an element of infinite order. Our hope is that the work we present here, together with the band and inverse semigroup cases, will lead to a better understanding of the homogeneity of completely regular semigroups. 

Henceforth, all structures considered will be of countable cardinality. The idempotents of a semigroup $S$ will be denoted by $E(S)$, and the identity of a group $G$ will be denoted by $\epsilon_G$, or simply $\epsilon$ if no confusion can occur. The identity automorphism of a structure $M$ is denoted by $\text{Id}_{M}$. Given a completely regular semigroup $S$ and $X\subseteq S$, we let $\langle X \rangle$ denote the completely regular semigroup generated by $X$. Notice if $S$ is a group, then $\langle X \rangle$ forms a subgroup of $S$.

\section{Basics of homogeneity} 

Our methods for proving homogeneity come in two forms: either we prove it directly from certain isomorphism theorems or we use the general method of Fra\"iss\'e. In this section we outline the latter method. Here we apply this only to completely simple semigroups (considered in the signature of unary semigroups), and for the general case we refer to \cite[Chapter 6]{Hodges97}. 

Let $\mathcal{K}$ be a class of finitely generated (f.g.) completely simple semigroups. Then we say:   
\begin{enumerate} [label=(\arabic*)] 
\item $\mathcal{K}$ is \textit{countable} if  it contains only countably many isomorphism types.
\item  $\mathcal{K} $ is \textit{closed under isomorphism} if whenever $A\in \mathcal{K}$ and $B\cong A$ then $B\in \mathcal{K}$.  
\item $\mathcal{K}$ has the \textit{hereditary property} (HP) if given $A\in \mathcal{K}$ and $B$ a f.g. completely simple subsemigroup of $A$ then $B\in \mathcal{K}$. 
\item  $\mathcal{K}$ has the \textit{joint embedding property} (JEP) if given $B_1,B_2\in \mathcal{K}$, then there exists $C\in \mathcal{K}$ and embeddings $f_i:B_i\rightarrow C$ ($i=1,2$).
\item  $\mathcal{K}$ has the \textit{amalgamation property}\footnote{This is also known as the \textit{weak amalgamation property}.}  (AP) if given $A, B_1, B_2\in \mathcal{K}$, where $A$ is non-empty, and embeddings $f_i:A\rightarrow B_i$ ($i=1,2$), then there exists $D\in \mathcal{K}$ and embeddings $g_i: B_i \rightarrow D$  such that 
\[     f_1 \circ g_1 = f_2 \circ g_2. 
\] 
The collection $A,B_1,B_2$ is known as an \textit{amalgam}, denoted by $[A;B_1,B_2].$\end{enumerate} 

The \textit{age} of a completely simple semigroup $S$ is the class of all f.g. completely simple semigroups which can be embedded in $S$. 

 Since the union of a chain of completely simple semigroups is itself completely simple, we may apply Fra\"iss\'e's Theorem \cite{Fraisse} to the case of completely simple semigroups as follows:  

\begin{theorem}[Fra\"iss\'e's Theorem for completely simple semigroups] Let  $\mathcal{K}$ be a non-empty countable class of f.g. completely simple semigroups which  is closed under isomorphism and satisfies HP, JEP and AP. Then there exists a unique, up to isomorphism, countable homogeneous completely simple semigroup $S$ such that $\mathcal{K}$ is the age of $S$. Conversely, the age of a countable homogeneous completely simple semigroup is closed under isomorphism, is countable and satisfies HP, JEP and AP. 
\end{theorem} 

We call $S$ the \textit{Fra\"iss\'e limit} of $\mathcal{K}$.

We note that the age of any structure can be seen to be closed under isomorphism and have HP and JEP (Fra\"iss\'e also showed the converse to hold). Consequently, to show that a structure is homogeneous it suffices to show that its age is countable and has AP.

\begin{example} Given a pair of index sets $I$ and $\Lambda$, we may form a band $B=I\times \Lambda$ with multiplication $(i,\lambda)(j,\mu)=(i,\mu)$. Then $B$ is a  \textit{rectangular band}, and is thus completely simple. The class of all finite rectangular bands forms a Fra\"iss\'e class, with Fra\"iss\'e limit the rectangular band $\mathbb{N} \times \mathbb{N}$ \cite{Quinnband}. 
\end{example} 

\section{homogeneous edge-coloured bipartite graphs} 

A major aim of this paper is to link the homogeneity of completely simple semigroups with previously studied homogeneous structures, chiefly groups and edge-coloured bipartite graphs. In this section we recap the work of Jenkinson, Seidel and Truss \cite{Jenkinson} on the homogeneity of edge-coloured bipartite graphs. Note that they considered only the case where the colouring set was finite, but it is necessary to give the background details in a more general setting. 

A \textit{bipartite graph} is a (simple) graph whose vertices can be split into two disjoint non-empty sets $L$ and $R$ such that every edge connects a vertex in $L$ to a vertex in $R$.
 The sets $L$ and $R$ are called the \textit{left set} and the \textit{right set}, respectively. We consider bipartite graphs in the signature $(E,L,R)$, where $E$ corresponds to the edge relation, and $L$ and $R$ are unary relations corresponding to the left and right sets, respectively. 
 
%  Formally, a bipartite graph is a triple $\Gamma=\langle L,R, E \rangle$ such that $L$ and $R$ are non-empty trivially intersecting sets and 
%\[ E\subseteq \{ \{x,y\}  :x\in L, \, y\in R \}.
%\] 
% We call $L\cup R$ the set of \textit{vertices} of $\Gamma$ and $E$ the set of \textit{edges}. An isomorphism between a pair of bipartite graphs $\Gamma=\langle L,R,E \rangle$ and $\Gamma'=\langle L',R',E' \rangle$ is a bijection $\psi: L\cup R \rightarrow L'\cup R'$ such that $L\psi=L'$,  $R\psi=R'$, and  $\{l,r\}\in E$ if and only if $\{l\psi,r\psi\}\in E'$.
 
A bipartite graph is called \textit{complete} if all vertices from $L$ and $R$ are joined by an edge. If each vertex  is incident to exactly one edge, then it is called a \textit{perfect matching}. The \textit{complement} of a bipartite graph $\Gamma$ is the bipartite graph with the same vertex set as $\Gamma$ but having precisely those edges which are not edges in $\Gamma$.  We call $\Gamma$ \textit{generic} if $|L|=\aleph_0=|R|$, and  for any pair of finite disjoint subsets $U$ and $V$ of $L$ (of $R$) there exists $x\in R$ ($x\in L$) joined to all elements of $U$ and to no elements of $V$.

We may colour the edges of a complete bipartite graph by colours from a non-empty set $C$, and we   call such a graph \textit{$C$-edge-coloured}, and the original case \textit{monochromatic}. 
Notice that a bipartite graph can be considered as a 2-edge-coloured bipartite graph, where the two colours correspond to ‘joined’ and ‘not joined’. Formally, we construct a colouring function $F$ from $L\times R$ to $C$, and a $C$-edge-coloured graph $\Gamma$ is considered in the signature $(E,L,R,E_1,\dots,E_{|C|})$, where $E_i$ is the binary relation corresponding to edges which are coloured by a fixed colour.

Our  choice of signature gives rise to a natural definition of isomorphism between edge-coloured bipartite graphs. Let $\Gamma=L\cup R$ and $\Gamma'=L'\cup R'$ be a pair of $C$-edge-coloured bipartite graphs. A bijection $\psi$ from $\Gamma$ to $\Gamma'$ is an \textit{isomorphism} if it preserves left and right sets (and thus edges), and preserves colours:  
\[ L\psi = L', \quad R\psi= R', \quad (x,y)F = (x\psi, y\psi)F. 
\] 
An (induced) sub-$C$-edge-coloured graph $A$ of $\Gamma$ is a subgraph $L'\cup R'$ of $\Gamma$ with each edge $(x,y)$ in $A$ coloured as in $\Gamma$. That is, the colouring function $F':L'\times R'$ to $C$ of $A$ is simply the restriction of the colouring function of $\Gamma$ to $A$. 

Given a colour set $C$, we say that the $C$-edge-coloured bipartite graph $\Gamma$ is \textit{$C$-generic} if $|L|=\aleph_0=|R|$ and for any map $\alpha$ from a finite subset of $L$ (of $R$) into $C$, there exists $x\in R$ ($x\in L$) such that for all $y\in$ dom $\alpha$, $(y,x)F=y\alpha$. It follows that there exist \textit{infinitely} many $x\in R$ with this property, and such elements are often referred to as \textit{witnesses}.

\begin{theorem}\label{thm:edge class} \cite{Jenkinson} If $\Gamma$ is a countable homogeneous $C$-edge-coloured bipartite graph where $1 \leq |C|
< \aleph_0$,  then one of  the following holds:
\begin{enumerate}
\item $|C|=1$ and all edges have the same colour,
\item  $|C| =2$ and the edges of one colour are a perfect matching, and those of the other colour are its
complement,
\item $|C | \geq  2$ and $\Gamma$ is $C$-generic.
\end{enumerate}
\end{theorem}

We note that the homogeneity of infinitely edge-coloured bipartite graphs was not considered in \cite{Jenkinson}. Fortunately, the only example of such a bipartite graph arising in this paper will be the $\omega$-generic bipartite graph $\Gamma$. The homogeneity of $\Gamma$ can be proved using Fra\"iss\'e's method, using an argument identical  to that used for the finite colouring case in  Lemma 2.1 of  \cite{Jenkinson}.  

\begin{lemma} (cf. \cite{Jenkinson}) Let $C$ be a (possibly infinite) colouring set $C$. Then the $C$-generic graph is the Fra\"iss\'e limit of the class of all finite bipartite $C$-edge-coloured graphs.
\end{lemma} 

\section{Morphisms between completely simple semigroups} 

Our hope of achieving a classification of homogeneous completely simple semigroups is aided by the well known structure theorem of Rees given below, as well as a relatively simple isomorphism theorem (Theorem \ref{iso thm}). We refer to \cite[Chapter 3]{Howie94} for a comprehensive study of completely simple semigroups and, in particular, a proof of Rees' Theorem. 

\begin{theorem}[Rees' Theorem] 
  Let $G$ be a group,  $I$ and $\Lambda$ be non-empty index sets and let $P=(p_{\lambda,i})$ be a $\Lambda \times I$ matrix with entries in $G$. Let $S=I\times G \times \Lambda$, and define multiplication on $S$ by 
\[  (i,g,\lambda)(j,h,\mu) = (i,g p_{\lambda, j} h,\mu) 
\]
Then $S$ is a completely simple semigroup, denoted by $\mathcal{M}[G;I,\Lambda;P]$. Conversely, every completely simple semigroup is isomorphic to a semigroup constructed in this way. 
\end{theorem}

We call $S=\mathcal{M}[G;I,\Lambda;P]$ a \textit{Rees matrix semigroup}.

\begin{remark} The triple $(G,I,\Lambda)$ arises from the \textit{Green's} relations  $\mathcal{L},\mathcal{R}$ and $\mathcal{H}=\mathcal{L}\cap \mathcal{R}$ of a completely simple semigroup $S$. The Green's relations form equivalence relations, and the $\mathcal{H}$-classes of $S$ are groups, which are pairwise isomorphic. The proof of Rees' Theorem  that $S\cong \mathcal{M}[G;S/\mathcal{R},S/\mathcal{L};P]$, where $G$ is isomorphic to the $\mathcal{H}$-classes of $S$ and $P$ is some matrix over $G$. 
%Define $x \, \mathcal{L} \, y$ if and only if $Sx=Sy$, $x \, \mathcal{R} \, y$ if and only if $xS=yS$, and $\mathcal{H}=\mathcal{L} \cap \mathcal{R}$.  
\end{remark} 

The following result follows immediately from \cite[Theroem 3.4.2]{Howie94} and Rees' Theorem: 

\begin{theorem} \label{thm: normal} Given a Rees matrix semigroup $S=\mathcal{M}[G;I,\Lambda;P]$ and any fixed elements $i\in I$, $\lambda\in \Lambda$,  there exists a $\Lambda\times I$ matrix $Q$ over $G$ with $q_{\lambda,j}=\epsilon=q_{i,\mu}$ for each $j\in I, \mu\in \Lambda$ and such that $S$ is isomorphic to $T=\mathcal{M}[G;I,\Lambda;Q]$.  
\end{theorem} 

We call $T$ the  \textit{normalisation} of $S$ along row $\lambda$ and column $i$. If $S$ is normalised, then we let $1_{\Lambda}\in \Lambda$  and $1_{I}\in I$ denote the row and column in which the normalisation has occurred.   The benefits of using the normalised form is highlighted in the following result. 

\begin{proposition} \cite{Howie76} Let $S=\mathcal{M}[G;I,\Lambda;P]$ be a normalised Rees matrix semigroup. Then 
\[ \langle E(S) \rangle = \mathcal{M}[\langle G^P \rangle;I,\Lambda;P]
\] 
 where $G^P=\{ p_{\lambda,i}:\lambda\in \Lambda,i\in I\}$. 
\end{proposition}

It is worth extending the notation of the proposition above. Given a $\Lambda\times I$ matrix $Q=(q_{\lambda,i})$ with entries from some set $X$, we denote by $X^Q$  the subset of $X$ given by 
\[ X^Q=\{q_{\lambda,i}: \lambda\in \Lambda,i\in I\}.
\] 
 The matrix $Q$ induces an $X^Q$-edge coloured graph, denoted by $\Gamma(Q)$, with left set  $\Lambda$, right set $I$, and colouring function $F: \Lambda \times I \rightarrow X^Q$ defined by 
 \[ (\lambda,i)F= q_{\lambda,i} \quad ( (\lambda,i)\in   \Lambda \times I).
\]

\begin{definition}\label{def: gamma} A normalised Rees matrix semigroup $S=\mathcal{M}[G;I,\Lambda;P]$ gives rise to two key bipartite graphs, one monochromatic, and one coloured:  

(1) We let $\Gamma_{P}$ denote the complete bipartite graph with left set  $\Lambda$ and right set $I$. 

(2) We denote by $\Gamma(S)$   the $G^{P'}$-edge-coloured graph $\Gamma(P')$, where $P'$ is the $\Lambda\setminus 1_{\Lambda}$ by  $I\setminus \{1_I\}$ submatrix of $P$. We call $\Gamma(S)$ the \textit{induced edge-coloured bipartite graph of $S$}. 
\end{definition} 

Notice that $\Gamma_{P}$ is always a homogeneous bipartite graph since it is complete. Our perhaps obscure choice of $\Gamma(S)$ will be justified in the next section, where we will show that the homogeneity of $S$ passes to $\Gamma(S)$ in the finite coloured case (that is, the case where $P$ has only finitely many distinct entries).

%\begin{lemma}\label{sub rees} Let $S=\mathcal{M}[G;I;\Lambda;P]$ be a completely simple semigroup, and let  $(i_1,g_1,\lambda_1),\dots,(i_n,g_n,\lambda_n)$ be a finite list of elements of $S$ such that $p_{\lambda_s,i_t}=e$ for each $1\leq s,t \leq n$. Then 
%\[ \langle (i_1,g_1,\lambda_1),\dots,(i_n,g_n,\lambda_n)\rangle = \mathcal{M}[\langle g_1,\dots,g_n\rangle;I';\Lambda';Q],
%\] 
%where $I'=\{i_1,\dots i_n\}$, $\Lambda'=\{\lambda_1,\dots,\lambda_n\}$, and  $Q$ is the $\Lambda ' \times I'$ submatrix of $P$ with entries $e$. 
%\end{lemma} 

As in \cite{Araujo10}, we adapt the isomorphism theorem for Rees matrix semigroups to explicitly highlight the role of the underlying bipartite graph: 

\begin{theorem}\label{iso thm} Let $S=\mathcal{M}[G;I,\Lambda;P]$ and $T=\mathcal{M}[H;J,M;Q]$ be a pair of normalised Rees matrix semigroups. Let $\theta:G\rightarrow H$ be a group morphism,   $\psi:\Gamma_{P} \rightarrow \Gamma_{Q}$ a bipartite graph morphism, and let $u_i,v_{\lambda}\in H$ ($i\in I, \lambda\in \Lambda$) be such that
\begin{equation*} \label{p eq} p_{\lambda,i}\theta=v_{\lambda} q_{\lambda\psi, i\psi} u_i
\end{equation*}
for all $i\in I$ and $\lambda\in \Lambda$. Define a map $\phi:S\rightarrow T$ given by 
\[ (i,g,\lambda)\phi= (i\psi, u_i(g\theta)v_{\lambda},\lambda\psi).
\]  
Then  $\phi$ is a morphism, denoted  by $[\theta,\psi,u_i,v_{\lambda}]$, and moreover every morphism from $S$ to $T$ can be constructed in this way. The morphism $\phi$ is injective/surjective if and only if both $\theta$ and $\psi$ are injective/surjective. 
\end{theorem}

In particular, if $S=\mathcal{M}[G;I,\Lambda;P]$ and $T=\mathcal{M}[H;J,M;Q]$ are isomorphic then $G\cong H$, $|I|=|J|$ and $|\Lambda|=|M|$. However, the morphisms $\theta$ and $\psi$, and the elements  $u_i,v_{\lambda}$ do not, in general, uniquely define the morphism $\phi$. This will become apparent in the following result.

Given a group $G$ and $u\in G$, we denote $C_u$ to be the inner automorphism of $G$ given by $gC_u=ugu^{-1}$. 

Given a structure $A$, we say that a substructure $B$ is  \textit{characteristic} if it is preserved by automorphisms of $A$, that is, if for all $\theta\in$ Aut($A$) we have $\theta(B)=B$. For example, for any  (completely simple) semigroup $S$, since automorphisms of $S$ map idempotents to idempotents, it follows that $\langle E(S) \rangle$ forms a characteristic (completely simple) subsemigroup. 

\begin{corollary} \label{cor:iso E} Let $S=\mathcal{M}[G;I,\Lambda;P]$ and $T=\mathcal{M}[H;J,M;Q]$ be a pair of normalised Rees matrix semigroups, and $\phi$ a morphism from $S$ to $T$. Then there are $u_i,v_{\lambda}\in \langle G^P \rangle$ such that $\phi=[\theta,\psi,u_i,v_{\lambda}]$, and $ [\theta|_{\langle G^P \rangle},\psi,u_i,v_{\lambda}]$ is a morphism from $\langle E(S) \rangle$ to $\langle E(T) \rangle$. In particular,  $\theta|_{\langle G^P \rangle}:\langle G^P \rangle\rightarrow \langle H^Q \rangle$, and  $u_i,v_{\lambda}\in \langle H^Q \rangle$. 
\end{corollary} 

\begin{proof} By Theorem \ref{iso thm} we may let $\phi=[\theta, \psi, {u}_i,{v}_{\lambda}]$. Since $\langle E(S) \rangle$ is a characteristic subsemigroup of $S$ it follows that $\phi'=\phi|_{\langle E(S) \rangle}$ is a morphism from $[\langle G^P \rangle;I,\Lambda;P]$ to $[\langle H^Q \rangle;J,M;Q]$. Suppose $\phi'=[\theta', \psi', \bar{u}_i,\bar{v}_{\lambda}]$. For each $(i,g,\lambda)\in \langle E(S) \rangle$, so that $g\in \langle G^P \rangle$,  we have 
\[ (i,g,\lambda)\phi = (i\psi,u_i(g\theta)v_{\lambda}, \lambda\psi)=(i\psi',\bar{u}_i(g\theta')\bar{v}_{\lambda},\lambda\psi')=(i,g,\lambda)\phi'
\] 
and so $\psi=\psi'$. By taking $g=\epsilon_G$ we have   $u_iv_{\lambda}=\bar{u}_i\bar{v}_{\lambda}$, so that $u_i^{-1} \bar{u}_i = v_{\lambda}\bar{v}_{\lambda}^{-1}$ for all $i\in I, \lambda\in \Lambda$.  
%\begin{equation} \label{eq:5} \bar{u}_{i}^{-1} u_i  = \bar{u}_{1_I}^{-1} u_{1_{I}} \text{ and } v_{\lambda}  \bar{v}_{\lambda}^{-1} = (\bar{u}_{1_{I}}^{-1}u_{1_{I}})^{-1}. 
%\end{equation} 
Letting $u=u^{-1}_{i}\bar{u}_{i}$ for any $i\in I$, we thus have
\[ (g\theta) = u^{-1}_i \bar{u}_i (g\theta') \bar{v}_{\lambda} v_{\lambda}^{-1} = u (g\theta') u^{-1}
\] 
and so $\theta|_{\langle G^P \rangle} = \theta'C_{u}$.

We claim that $\phi = [\theta C_{u^{-1}}, \psi,\bar{u}_i,\bar{v}_{\lambda}]$. If $g\in G$ then 
\begin{align*}  \bar{u}_i (g\theta C_{u^{-1}}) \bar{v}_{\lambda} & = \bar{u}_i (u^{-1}) (g\theta)( u) \bar{v}_{\lambda}
 \\ & = \bar{u}_i (\bar{u}_i^{-1}u_i) (g\theta) (v_{\lambda}  \bar{v}_{\lambda}^{-1}) \bar{v}_{\lambda}
   = u_i (g\theta) v_{\lambda},
\end{align*} 
thus proving the claim. The result then follows as $\theta C_{u^{-1}}=\theta C_{u}^{-1}$ extends $\theta'$.
\end{proof}

The proof above may be  adapted to show exactly when a pair of morphisms between Rees matrix semigroups are equal: 

\begin{corollary}\label{lemma: morph equiv}  Let $S=\mathcal{M}[G;I,\Lambda;P]$ and $T=\mathcal{M}[H;J,M;Q]$ be a pair of Rees matrix semigroups, and  $\phi=[\theta,\psi,u_i,v_{\lambda}]$ and  $\phi'=[\theta',\psi',u_i',v_{\lambda}']$ be a pair of morphisms from $S$ to $T$. Then $\phi=\phi'$ if and only if  $\theta=\theta'C_{u_1^{-1}u_1'}$, $\psi=\psi'$, and $u_iv_{\lambda} = u'_i v_{\lambda}'$ for all $i\in I, \lambda\in \Lambda$. 
\end{corollary} 

%\begin{proof} $(\Rightarrow)$ For each $i\in I, \lambda\in \Lambda$ and $g\in G$, notice that 
%\[ (i,g,\lambda)\phi = (i\psi,u_i(g\theta)v_{\lambda}, \lambda\psi)=(i\psi',u_i'(g\theta')v_{\lambda}',\lambda\psi')=(i,e,\lambda)\phi'
%\] 
%and so $\psi=\psi'$. Moreover, by taking $g=e_G$ we have $u_iv_{\lambda} = u'_i v_{\lambda}'$, so that $u^{-1}u_i' =(v_{\lambda}'v^{-1})^{-1}$. Hence, 
%\[ (g\theta) = u^{-1}_i u_i' (g\theta') v_{\lambda}'v_{\lambda}^{-1} = u^{-1}_1u_1' (g_{\lambda}' (u^{-1}u_1')^{-1}
%\] 
%and so $\theta=\theta'C_{u_1^{-1}u_1'}$. 
%
%$(\Leftarrow)$ For any $(i,g,\lambda)\in S$ we have 
%\[ u_i(g\theta)v_{\lambda} = u_i\big(u_1^{-1}u_1' (g\theta') u_1^{'-1} u_1 \big)v_{\lambda} =  u_i\big(u_i^{-1}u_i' (g\theta') v_{\lambda}'v_{\lambda}^{-1} \big)v_{\lambda} = u_i' (g\theta') v_{\lambda}'
%\] 
%from which the result follows. 
%\end{proof} 
To make use of the induced edge-coloured bipartite graph of a Rees matrix semigroup, we need to consider those morphisms which fix the normalised row and column: 

\begin{corollary}\label{cor: inner} Let $S=\mathcal{M}[G;I,\Lambda;P]$ and $T=\mathcal{M}[H;J,M;Q]$ be a pair of normalised Rees matrix semigroups, and $\phi=[\theta,\psi,u_i,v_{\lambda}]$ a morphism from $S$ to $T$ such that $1_I\psi = 1_J$ and $1_\Lambda=1_M$. Then there exists $u\in H$ such that $u_i=u$ and $v_{\lambda}=u^{-1}$ for all $i\in I$ and $\lambda\in \Lambda$, and $\phi=[\theta C_u, \psi, \epsilon_H,\epsilon_H]$. 
 Moreover, $(1_I,p_{\lambda,i},1_{\Lambda})\phi=(1_J,p_{\lambda\psi,i\psi},1_M)$ for any $p_{\lambda,i}\in G^P$. 
\end{corollary} 

\begin{proof} For each $\lambda\in \Lambda$ we have, by Theorem \ref{iso thm},
\[ \epsilon_H = p_{\lambda,1_I}\theta =v_{\lambda} p_{\lambda\psi,1_I\psi} u_{1_I} = v_{\lambda} p_{\lambda\psi,1_I} u_{1_I} = v_{\lambda}u_{1_I}
\]
so that $v_{\lambda}=u_{1_I}^{-1}$.  Dually, $v_{1_{\Lambda}}u_i=\epsilon_H$ for all $i\in I$, and so $u_i=v_{1_{\Lambda}}^{-1}=u_{1_I}$. Let $u=u_{1_I}$. Then for each $i\in I, \lambda \in \Lambda$ we have $u_iv_{\lambda} = u u^{-1}=\epsilon_H$ and 
\[ (\theta C_u)C_{u^{-1}\epsilon_H} = (\theta C_u) C_{u^{-1}} = \theta,
\] 
so that $\phi = [\theta C_u, \psi, \epsilon_H,\epsilon_H]$ by Corollary \ref{lemma: morph equiv}. 

Finally, for any $i\in I, \lambda\in \Lambda$, 
\[ (1_I,p_{\lambda,i},1_{\Lambda})\phi= (1_I,\epsilon_H (p_{\lambda,i}\theta) \epsilon_H, 1_\Lambda), 
\] 
and $p_{\lambda,i}\theta =\epsilon_H p_{\lambda\psi,i\psi}\epsilon_H=p_{\lambda\psi,i\psi}$ by Theorem \ref{iso thm}.
\end{proof} 

\begin{corollary} \label{iso orth} Let $S=\mathcal{M}[G;I,\Lambda;P]$ and $T=\mathcal{M}[H;J,M;Q]$ be a pair of Rees matrix semigroups with $P$ and $Q$ matrices over $\{\epsilon_G\}$ and $\{\epsilon_H\}$, respectively. Let $\theta:G\rightarrow H$ and  $\psi:\Gamma_P \rightarrow \Gamma_Q$ be morphisms. Then $\phi=[\theta,\psi,\epsilon_H,\epsilon_H]$ is a morphism from $S$ to $T$, and moreover every morphism can be constructed this way. 
\end{corollary} 

\begin{proof} Immediate from the proof of Corollary \ref{cor: inner}. 
\end{proof}

\section{Homogeneity of completely simple semigroups} 

In this section we classify  homogeneous completely simple semigroups, up to the determination of the homogeneous groups. Given that we now better understand isomorphisms between completely simple semigroups, the next step is to construct   f.g.  completely simple semigroups. The following lemma is folklore, and is easily verified: 

\begin{lemma} A Rees matrix semigroup $\mathcal{M}[G;I,\Lambda;P]$  is f.g. if and only if $G$ is f.g. and both $I$ and $\Lambda$ are finite. 
\end{lemma} 

A subsemigroup $T$ of $S=\mathcal{M}[G;I,\Lambda;P]$ is called a \textit{Rees subsemigroup} if there exists a subgroup $H$ of $G$, $J\subseteq I$ and $M\subseteq \Lambda$ such that $T=\mathcal{M}[H;J,M;Q]$, where $Q$ is the $M\times J$ submatrix of $P$. Note that not every completely simple subsemigroup of $S$ is of this form, as shown in \cite{Tian}. 

This section will build towards a proof of the following theorem, which links the homogeneity of a completely simple semigroup $\mathcal{M}[G;I,\Lambda;P]$ to that of the group $G$ and the subsemigroup $\langle E(S) \rangle$: 

\begin{theorem}\label{thm:S iff E,G}   Let $S= \mathcal{M}[G;I,\Lambda;P]$ be a  completely simple semigroup. Then $S$ is homogeneous if and only if $G$ and $\langle E(S) \rangle$ are homogeneous, and the set $G^P$ forms a characteristic subgroup of $G$. 
\end{theorem} 

The forward direction of the theorem above is proved in the next result, together with Proposition \ref{GP char}. The backwards direction will follow from the classification theorem for homogeneous completely simple semigroups (Theorem \ref{thm: classify}). 

\begin{proposition} \label{prop:G, E hom} Let $S= \mathcal{M}[G;I,\Lambda;P]$ be a homogeneous completely simple semigroup. Then $G$ and $\langle E(S) \rangle$ are homogeneous. 
\end{proposition} 

\begin{proof} Since $\langle E(S) \rangle$ forms a characteristic subsemigroup of $S$, it is clear that the homogeneity of $S$ passes to $\langle E(S) \rangle$. 
 If $\theta:H\rightarrow K$ is an isomorphism of f.g. subgroups of $G$, then the map 
\[ \phi:\{(1_I,g,1_\Lambda):g\in H\} \rightarrow \{(1_I,h,1_{\Lambda}):h\in K\}, \quad (1_I,g,1_\Lambda)\phi=(1_I,g\theta,1_\Lambda)
\] 
is clearly an isomorphism between f.g. subsemigroups of $S$. Extending $\phi$ to an automorphism $\phi'=[\theta',\psi,\epsilon,\epsilon]$ of $S$ (noting our use of Corollary \ref{cor: inner}) then for any $g\in H$, 
\[ (1_I,g,1_\Lambda)\phi=(1_I,g\theta,1_\Lambda)=(1_I,g\theta',1_\Lambda)=(1_I,g,1_\Lambda)\phi'
\] 
and so $\theta'$ extends $\theta$ as required. 
\end{proof} 

Theorem \ref{thm:S iff E,G} will now be  shown to hold when we place a strong restriction on our sandwich matrix. This result will prove vital for characterising homogeneous Rees matrix semigroups with finite sandwich matrices. 

\begin{theorem} \label{thm: E,G} Let $S =\mathcal{M}[G;I,\Lambda;P]$ be a normalised Rees matrix semigroup where the set $G^P$ forms a simple abelian group. Then $S$ is homogeneous if and only $G$ and $\langle E(S) \rangle$ are homogeneous, and the set $G^P $ is a characteristic subgroup of $G$.
\end{theorem} 

\begin{proof}
 Suppose $G$ and $E=\langle E(S) \rangle = [ G^P ;I,\Lambda;P]$ are homogeneous, with $ G^P$ forming a characteristic subgroup of $G$.  Since $G^P$ is a simple abelian group, it is either trivial or isomorphic to $\mathbb{Z}_p$ for some prime $p$, so age$( G^P )=\{\{\epsilon\},\mathbb{Z}_p\}$ up to isomorphism. 
To prove the homogeneity of $S$ it suffices, by Fra\"iss\'e's Theorem, to show that age($S$) is countable and has the AP. 

First, note that the isomorphism types of age($S$) are completely determined by age($G$) and age($\langle E(S)\rangle$) by Corollary \ref{cor:iso E}. 
Hence age($S$) is countable since both age($G$) and age($\langle E(S)\rangle$) are. 

Let $[M_0;M_1,M_2]$ be an amalgam in age($S$), where $M_k=\mathcal{M}[H_k;I_k,\Lambda_k;P_k]$ $(k=0,1,2)$. As in \cite{Clarke} we may assume that $H_1\cap H_2=H_0$, $I_1\cap I_2 = I_0$ and $\Lambda_1\cap \Lambda_2=\Lambda_0$. Moreover, we may assume that each $M_k$ is normalised, and by normalising via some $i\in I_0$ and $\lambda\in \Lambda_0$, we may assume  that $1_{I_1}=1_{I_0}=1_{I_2}=1$ and $1_{\Lambda_1}=1_{\Lambda_0}=1_{\Lambda_2}=1'$ by Theorem \ref{thm: normal}. 

For each $k$, let $H_k'=\langle H_k^{P_k} \rangle$, so that $E_k=\langle E(M_k) \rangle = \mathcal{M}[H_k';I_k,\Lambda_k;P_k]$ ($k=0,1,2$). Then $[E_0;E_1,E_2]$ is an amalgam in $E$, and so by the homogeneity of $E$ there exists a pair of embeddings $\phi_k=[\theta_k,\psi_k,u_i^{(k)},v_{\lambda}^{(k)}]: E_k \rightarrow E'=\mathcal{M}[A;I',\Lambda';P']\in \text{age}(E)$ ($k=1,2$) such that $\phi_1 =  \phi_2$ on $E_0$ (where we assume $E'$ is normalised). 
Note that inner automorphisms of $A$ are trivial as $A\in \text{age}(G^P)$ is abelian, and it thus follows from Corollary \ref{lemma: morph equiv} that  $\phi_2 = [\theta_2,\psi_2,\bar{u}_i,\bar{v}_{\lambda}]$ where
\[ \bar{u}_i = (u_1^{(1)} (u_1^{(2)})^{-1}) u_i^{(2)} \quad \text{and} \quad \bar{v}_{\lambda}=(u_1^{(1)} (u_1^{(2)})^{-1})^{-1} v_{\lambda}^{(2)}. 
\] 
We may thus assume without loss of generality that $u_1^{(1)} = u_1^{(2)}$.

 For any $(i,g,\lambda)\in E_0$ we have 
 \[ (i,g,\lambda)\phi_1 = (i\psi_1,u_i^{(1)} (g\theta_1) v_{\lambda}^{(1)}, \lambda\psi_1) = (i\psi_2,u_i^{(2)} (g\theta_2) v_{\lambda}^{(2)}, \lambda\psi_2) = (i,g,\lambda)\phi_2
 \] 
 and so $\psi_1=\psi_2$ on $\Gamma_{P_0}$ and $u_i^{(1)}v_{\lambda}^{(1)} = u_i^{(2)}v_{\lambda}^{(2)}$, and by the usual argument $\theta_1=\theta_2 C_x=\theta_2$ on $H_0'$, where $x=(u_{1}^{(1)})^{-1} u_{1}^{(2)}$. Moreover, since $u_1^{(1)}v_{\lambda}^{(1)} = u_1^{(2)}v_{\lambda}^{(2)}$ for any $\lambda\in \Lambda_0$  we have  $v_{\lambda}^{(1)} = v_{\lambda}^{(2)}$. Hence   $u_i^{(1)} = u_i^{(2)}$ for any $i\in I_0$.  
 
Since $[H_0;H_1,H_2]$ is an amalgam in age($G$) there exists a pair of embeddings $\varphi_k:H_k\rightarrow K\in \text{age}(G)$ ($k=1,2$) such that $\varphi_1=\varphi_2$ on $H_0$. Further, as age($G$) has the JEP  we may assume without loss of generality that $K$ contains a copy of $A$. 

Note that if $H_k'$ is non-trivial for some $k=1,2$, then $H_k'\cong \mathbb{Z}_p$, and so $\theta_k$ is an isomorphism. 
 Let $\chi:A\rightarrow K$ be the embedding given by  
\[ \chi = \begin{cases} \theta_1^{-1}\varphi_1, & \mbox{if } H_1'\neq \{\epsilon\}, \\ \theta_2^{-1}\varphi_2, & \mbox{if } H_2'\neq \{\epsilon\}, \\
\mbox{any embedding} & \mbox{otherwise}, \end{cases}
\] 
  noting that $\chi$ is well defined as $\theta_1^{-1}\varphi_1 = \theta_2^{-1}\varphi_2$ if both $H_1'$ and $H_2'$ are non-trivial. Consider the Rees matrix semigroup $M'=\mathcal{M}[K;I',\Lambda';P^*]$, where $p_{\lambda,i}^* = p_{\lambda,i}'\chi$. For each $k=1,2$, let  $\phi_k'=[\varphi_k,\psi_k,u_i^{(k)}\chi, v_{\lambda}^{(k)}\chi]$. We claim that $\phi_k'$ is an embedding of $M_k$ into $M'$. For any $i\in I_k, \lambda\in \Lambda_k$,  
\begin{equation}\label{eq 5}  p_{\lambda,i}^{(k)} \theta_k = v_{\lambda}^{(k)} p'_{\lambda\psi_k,i\psi_k} u_i^{(k)},
\end{equation} 
as $\phi_k$ is a morphism,  and so by applying $\chi$ we have 
  \[ (p_{\lambda,i}^{(k)} \theta_k)\chi = (v_{\lambda}^{(k)}\chi) p^*_{\lambda\psi_k,i\psi_k} (u_i^{(k)}\chi).
  \] 
  If $H_k'=\{\epsilon\}$ ($k=1,2$) then $ p_{\lambda,i}^{(k)}=\epsilon$, so that $p_{\lambda,i}^{(k)} \varphi_k = \epsilon = p_{\lambda,i}^{(k)}\theta_k \chi$. Otherwise, $\chi = \theta_k^{-1}\varphi_k$, so that
  \[ (p_{\lambda,i}^{(k)} \theta_k)\chi = (p_{\lambda,i}^{(k)} \theta_k)\theta_k^{-1}\varphi_k = p_{\lambda,i}^{(k)} \varphi_k
  \] 
  thus completing our claim. It thus suffices to prove that $\phi_1'$ and $\phi_2'$ agree on $M_0$. Let $(i,g,\lambda)\in M_0$. Then as $\psi_1=\psi_2$ on $\Gamma_{P_0}$ it in turn suffices to prove that
\begin{equation} \label{eq 6} (u_i^{(1)}\chi)(g\varphi_1)(v_{\lambda}^{(1)}\chi) = (u_i^{(2)}\chi)(g\varphi_2)(v_{\lambda}^{(2)}\chi). 
  \end{equation} 
 Since  $i\in I_0$ and $\lambda\in\Lambda_0$ we have that $u_i^{(1)} = u_i^{(2)}$ and $v_{\lambda}^{(1)} = v_{\lambda}^{(2)}$. Moreover,  $\varphi_1=\varphi_2$ on $H_0$, and so  \eqref{eq 6} holds as required.  
\end{proof}

The direct product of a group and a rectangular band is called a \textit{rectangular group}. A semigroup $S$ in which $E(S)$ forms a subsemigroup is called \textit{orthodox}.  It then holds that a semigroup $S$ is isomorphic to a rectangular group if and only if $S$ is an orthodox completely simple semigroup or,  equivalently, if $S$ is isomorphic to a Rees matrix semigroup in which the sandwich matrix contains only the identity element \cite[Section 3.2]{Clif&Pres61}. Consequently, the result above holds for rectangular groups. 

Every rectangular band was shown to be homogeneous by the author in \cite{Quinnband}, and the result extends to rectangular groups as follows: 

\begin{theorem}\label{orth} Let $S=G \times B$ be a rectangular group, where $G$ is a group and $B$ is a rectangular band.
 Then $S$ is  homogeneous if and only if $G$ is a homogeneous group. 
\end{theorem}  

\begin{proof} We may assume  that $S=\mathcal{M}[G;I,\Lambda;P]$, where  $p_{\lambda,i}=\epsilon$ for all $\lambda,i$. Hence $E(S)=\langle E(S) \rangle = \mathcal{M}[\{\epsilon\};I,\Lambda;P]$, with $E(S)$ being isomorphic to $B$, and thus homogeneous.  Hence if $G$ is a homogeneous group it follows immediately from Theorem \ref{thm: E,G} that $S$ is homogeneous. The converse is from Proposition \ref{prop:G, E hom}. 
% Let $S_1= \langle (i_1,g_1,\lambda_1), \dots, (i_n,g_n, \lambda_n) \rangle$ and $S_2= \langle (j_1,h_1,\mu_1), \dots, (j_m,h_m,\mu_m) \rangle$ be subsemigroups of $S$ and let $\phi:S_1\rightarrow S_2$ be an isomorphism. Then by Lemma \ref{sub rees} and Corollary \ref{iso orth} we have that
%\begin{align*} & S_1=\mathcal[\langle g_1,\dots,g_n \rangle; \{i_1,\dots,i_n\}, \{\lambda_1,\dots,\lambda_n\};P_1], \\
%& S_2= \mathcal[\langle h_1,\dots,h_m \rangle; \{j_1,\dots,j_m\}, \{\mu_1,\dots,\mu_m\};P_2],
%\end{align*} 
%and that $\phi=[\theta,\psi,\epsilon,\epsilon]$ for some  isomorphism $\theta:\langle g_1,\dots,g_n \rangle \rightarrow \langle h_1,\dots,h_m \rangle$, and $\psi:\Gamma_{P_1}\rightarrow \Gamma_{P_2}$. By the homogeneity of $G$ and $\Gamma_P$ we may extend $\theta$ and $\psi$ to a pair of automorphisms $\hat{\theta}$ and $\hat{\psi}$ of $G$ and $\Gamma_P$, respectively. Then $[\hat{\theta},\hat{\psi},\epsilon,\epsilon]$ is an automorphism of $S$, and extends $\phi$ as required. 
\end{proof}

We now consider the homogeneity of a non-orthodox Rees matrix semigroup $S=\mathcal{M}[G;I,\Lambda;P]$, so that $\Gamma(S)$ (Definition \ref{def: gamma}) is not coloured by a single colour. We first show that the homogeneity of $S$ passes to $\Gamma(S)$ when $\Gamma(S)$ is finitely coloured. The proof requires the following simple consequence of Theorem \ref{iso thm}. 

\begin{lemma} \label{Lem extend} Let $S= \mathcal{M}[G;I,\Lambda;P]$ and $T=\mathcal{M}[G;J,M;Q]$ be a pair of normalised Rees matrix semigroups over a group $G$. 
Let $\psi:\Gamma(S)\rightarrow \Gamma(T)$ be an isomorphism. Then the map $[\text{Id}_G,\hat{\psi},\epsilon_G,\epsilon_G]$ is an isomorphism from $S$ to $T$, where $\hat{\psi}$ extends $\psi$ with $1_I\hat{\psi} = 1_J$, $1_\Lambda \hat{\psi} = 1_M$. 
\end{lemma} 

\begin{proof} The proof is immediate, as $S$ and $ T$ are normalised, and  $p_{\lambda,i} = p_{\lambda\psi, i\psi}$ for each $\lambda,i\in \Gamma(S)$ as $\psi$ preserves colours. 
\end{proof}

\begin{proposition}\label{prop: graph hom} Let  $S=\mathcal{M}[G;I,\Lambda;P]$ be a homogeneous completely simple semigroup with $G^P$ finite. Then $\Gamma(S)$ is homogeneous. 
\end{proposition} 

\begin{proof} Let $\Gamma_k$ ($k=1,2$) be a pair of f.g. sub-edge-coloured graphs of $\Gamma(S)$, with left sets $I_k$ and right sets $\Lambda_k$. 
Let $\psi:\Gamma_1\rightarrow \Gamma_2$ be an isomorphism (as $G^{P'}$-edge-coloured graphs). Let $P_k$ be the  $\{1_{\Lambda}\} \cup \Lambda_k$ by $\{1_I\}\cup I_k$ submatrix of $P$.  Then  $S_k=\mathcal[\langle G^P \rangle;I_k,\Lambda_k;P_k]$ ($k=1,2$) are a pair of normalised Rees matrix subsemigroups of $S$, and are  f.g. as $G^P$ is finite. Moreover, by Lemma \ref{Lem extend}, $\phi=[\text{Id}_{\langle G^P \rangle},\hat{\psi}, \epsilon,\epsilon]$ is an isomorphism from $S_1$ to $S_2$, where $\hat{\psi}$ extends $\psi$ by fixing $1_I$ and $1_\Lambda$.
 By the homogeneity of $S$ and Corollary \ref{cor: inner} we may extend  $\phi$ to an automorphism $\phi'=[\theta,\varphi,\epsilon,\epsilon]$ of $S$. For each $p_{\lambda,i}\in G^P$, since $(1_{I},p_{\lambda,i},1_\Lambda)$ is fixed by $\phi$, and thus by $\phi'$, we have from Corollary \ref{cor: inner} that $p_{\lambda\varphi,i\varphi}=p_{\lambda,i}$. Hence $\varphi'=\varphi|_{\Gamma(S)}$ is an automorphism of $\Gamma(S)$, from which the result follows. 
\end{proof}

\begin{lemma} \label{finite left} Let $S=\mathcal{M}[G;I,\Lambda;P]$ be a homogeneous normalised Rees matrix semigroup. Let $H$ be a finite subset of $G^P$, $J$ a finite subset of $I$ containing $1_I$, and $\psi$ a bijection of $J$ fixing $1_I$.  Then there exists an automorphism $\varphi$ of $\Gamma_{P}$ extending $\psi$, with $1_\Lambda\varphi=1_{\Lambda}$ and such that $p_{\lambda,i}=p_{\lambda\varphi,i\varphi}$ for all $p_{\lambda,i}\in H$. Dually for finite subsets of $\Lambda$. 
\end{lemma} 

\begin{proof}  Consider the f.g. Rees subsemigroup of $S$ given by $T=\mathcal{M}[\langle H \rangle;J,\{1_{\Lambda}\};Q]$. Let $\psi'$ be the automorphism of $\Gamma_{Q}$ which  fixes $1_{\Lambda}$ and such that $j\psi'=j\psi$ for each $j\in J$. 
Then $\phi=[Id_{\langle H \rangle}, \psi',\epsilon, \epsilon]$ is an automorphism of $T$ by Corollary \ref{iso orth}. 	Extend $\phi$ to an automorphism $\phi'$ of $S$, noting that as $\psi'$ fixes $1_I$ and $1_\Lambda$ we may assume $\phi'=[\theta,\varphi,\epsilon,\epsilon]$ by Corollary \ref{cor: inner}. Let $p_{\lambda,i}\in H$. Then $(1_{I},p_{\lambda,i},1_\Lambda)$ is fixed by $\phi$ and so  $p_{\lambda,i}=p_{\lambda\varphi,i\varphi}$ by Corollary \ref{cor: inner} as required. 
\end{proof}

We let $C(i)=\{p_{\lambda,i}:\lambda\in \Lambda\}$ denote the entries of $P$ in column $i$, and $R(\lambda)=\{p_{\lambda,i}:i\in I\}$ denote the entries of $P$ in row $\lambda$. Unless stated otherwise, we let $I'=I\setminus \{1_I\}$, $\Lambda'=\Lambda\setminus \{1_\Lambda\}$, and $P'$ be the $\Lambda' \times I'$ submatrix of $P$. 

\begin{corollary} \label{column same} Let $S=\mathcal{M}[G;I,\Lambda;P]$ be a homogeneous normalised Rees matrix semigroup. Then $G^P=C(i)=R(\lambda)$ for any $i \in I'$ and any $\lambda\in \Lambda'$. 
\end{corollary} 
\begin{proof}
 We prove that $C(i)=C(j)$ for any $i,j\in I'$ and $R(\lambda)=R(\mu)$ for any $\lambda,\mu \in \Lambda'$, from which the result is immediate. Let $i,j\in I'$ and take any $p_{\lambda,i}\in C(i)$.  Let $H=\{p_{\lambda,i}\}$, $J=\{1_I,i,j\}$ and $\psi$ be a bijection of $J$ fixing $1_I$ and swapping $i$ and $j$. Then by Lemma \ref{finite left} there exists an automorphism $\varphi$ of $\Gamma_P$ such that $p_{\lambda,i} = p_{\lambda\varphi,j}\in C(j)$. Hence $C(i)\subseteq C(j)$, and a similar argument gives equality. Dually for rows. 
\end{proof}

\begin{proposition}   Let $S=\mathcal{M}[G;I,\Lambda;P]$ be a homogeneous normalised Rees matrix semigroup. Then $I$ is finite if and only if $\Lambda$ is finite. 
\end{proposition} 

\begin{proof} Suppose $I$ is finite. Then by Corollary \ref{column same}  we have that $G^P$ is finite, and so $\Gamma(S)$ is homogeneous by Proposition \ref{prop: graph hom}. It then follows Theorem \ref{thm:edge class} that $\Lambda$ is finite. Dually for $\Lambda$.  
\end{proof}

Consequently, the sandwich matrix of a homogeneous Rees matrix semigroup is either finite, or is  infinite by infinite.
  We are now able to complete our proof of the forward direction of Theorem \ref{thm:S iff E,G}. 

\begin{proposition}\label{GP char}  Let $S=\mathcal{M}[G;I,\Lambda;P]$ be a homogeneous normalised Rees matrix semigroup. Then $G^P$ is a characteristic subgroup of $G$. 
\end{proposition} 

\begin{proof}
% We first show that if $G^P$ contains an element of order $n$, then it contains every element of $G$ of order $n$. Let $p_{\lambda,i}$ have order $n\in \mathbb{N}^*=\mathbb{N}\cup \{\aleph_0\}$, and $x\in G$ be any element of order $n$. Then  the map 
%\[ \phi:\langle (1_I,p_{\lambda,i},1_\Lambda) \rangle \rightarrow \langle(1_I,x,1_\Lambda) \rangle, \quad (1_I,p_{\lambda,i}^m,1_\Lambda)\phi=(1_I,x^m,1_{\Lambda}) \quad (m\in \mathbb{Z})
%\] 
% is an isomorphism. We may extend $\phi$ to an automorphism $[\theta,\psi,\epsilon,\epsilon]$ by the homogeneity of $S$ and Corollary \ref{cor: inner}.
% Notice that  $(1_I,p_{\lambda,i},1_{\Lambda})\phi =(1_I,x,1_\Lambda)$, so that $p_{\lambda\psi,i\psi}=x$ by Corollary \ref{cor: inner} as required. Consequently $G^P$ is a characteristic subset of $G$. 

Let $a,b\in G^P$, so that by Lemma \ref{column same} we may assume that $a=p_{\lambda,i}$ and $b=p_{\mu,i}$ for some $i\in I$ and  $\lambda,\mu\in \Lambda$. Letting $H=\langle a,b \rangle$, consider a pair of f.g. Rees  subsemigroups of $S$ given by $S_1=\mathcal{M}[H;\{1,i\},\{\lambda\};P_1]$ and $S_2=\mathcal{M}[H;\{1,i\},\{\mu\};P_2]$.
 Let $\psi:\Gamma_{P_1}\rightarrow \Gamma_{P_2}$ be the isomorphism fixing $1_I$ and $i$, and with $\lambda\psi=\mu$. Let $\bar{u}_{1_I}=\bar{v}_\lambda = \epsilon$ and $\bar{u}_i=b^{-1}a$. Then it is a simple exercise to show that $\phi=[Id_{H},\psi,\bar{u}_i,\bar{v}_{\lambda}]$ is an isomorphism from $S_1$ to $S_2$, which we may thus extend to $\phi'=[\theta,\psi',{u}_i,{v}_\lambda]\in \text{Aut}(S)$. For each $\tau\in \Lambda$ we have 
 \[ \epsilon=p_{\tau,1_I}\theta = v_{\tau} p_{\tau\psi',1_I} u_{1_I} = v_{\tau}u_{1_I},
 \] 
  and so $v_{\tau}=u_{1_I}^{-1}$. By considering the image of $(i,\epsilon,\lambda)$ by $\phi$ and $\phi'$ we have that $u_i u_{1_I}^{-1}= b^{-1}a$. Let $1_{\Lambda}\psi=\sigma$, so that 
\[ \epsilon=p_{1_\Lambda,i}\theta= u_{1_I}^{-1} p_{\sigma,i} u_i.
\] 
Then $p_{\sigma,i} = u_{1_I}u_i^{-1} = a^{-1}b\in G^P$, and hence $G^P$ is a subgroup of $G$. 

Now let $\theta$ be an automorphism of $G$, and let $p_{\lambda,i}\in G^P$. Then the map 
\[ \phi:\langle (1_I,p_{\lambda,i},1_\Lambda) \rangle \rightarrow \langle(1_I,p_{\lambda,i}\theta,1_\Lambda) \rangle, \quad (1_I,p_{\lambda,i}^m,1_\Lambda)\phi=(1_I,(p_{\lambda,i}\theta)^m,1_{\Lambda}) \quad (m\in \mathbb{Z})
\] 
is an isomorphism. We may extend $\phi$ to an automorphism $[\chi,\psi,\epsilon,\epsilon]$ by the homogeneity of $S$ and Corollary \ref{cor: inner}.  
Notice that  $(1_I,p_{\lambda,i},1_{\Lambda})\phi =(1_I,p_{\lambda,i}\theta,1_\Lambda)$, so that $p_{\lambda\psi,i\psi}=p_{\lambda,i}\theta\in G^P$ by Corollary \ref{cor: inner}. Hence $G^P$ is a characteristic subgroup of $G$. 
\end{proof} 

%\begin{theorem} Let $S =\mathcal{M}[G;I;\Lambda;P]$ be a Rees matrix semigroup. Then $S$ is homogeneous if and only $G$ and $\langle E(S) \rangle$ are homogeneous, and $G(P)$ forms a characteristic subgroup of $G$.
%\end{theorem} 
%
%\begin{proof} Suppose $G$ and $\langle E(S) \rangle = [G(P);I;\Lambda;P]$ are homogeneous, with $G(P)$ characteristic in $G$. To obtain our result we shall use Fraisse's Theorem by showing that Age($S$) has amalgamation. 
%
%Let $[M_0;M_1,M_2]$ be an amalgam in Age($S$), where $M_k=\mathcal{M}[H_k;I_k;\Lambda_k;P_k]$ $(k=0,1,2)$. We may assume without loss of generality that each $M_k$ is normalised, and let $H_k'=\langle H_k(P_k) \rangle$.
% Let $\phi_k=[\theta_k,\psi_k,u_i^{k},v_{\lambda}^k]$ be a pair of embeddings ($k=1,2$), where we may assume that $\phi_k'=[\theta_k|_{H_k'}, \psi_k,u_i^{(k)},v_\lambda^{(k)}]$ is an embedding of $E_0=\langle E(M_0) \rangle$ to $E_k=\langle E(M_k) \rangle$ by ???. 
% Hence $[E_0;E_1,E_2]$ is an amalgam in Age$(\langle E(S) \rangle)$, and so there exists (normalised) $E'=[H';I';\Lambda',P']\in$ Age$(\langle E(S) \rangle)$ and embeddings $\varphi_k=[\delta_k,\omega_k,\bar{u}_i^k,\bar{v}_{\lambda}^{(k)}]:E_k\rightarrow E'$ such that  $\phi_1'\varphi_1 = \phi_2'\varphi_2$. 
%\end{proof}
%
Our classification naturally splits into two cases, based on whether $G^P$ is finite or not. However, it will be easier to simultaneously consider  the cases where $\Gamma(S)$ is of generic type  or $G^P$ is infinite. 

\subsection{$\Gamma(S)$ finitely coloured and not of generic type} 

In this subsection  we classify the homogeneity of Rees matrix semigroups $S=\mathcal{M}[G;I,\Lambda;P]$ where $\Gamma(S)$ is finitely coloured, so that $G^P$ is finite, but not of generic type. By Theorem \ref{thm:edge class} $P'=(\Lambda\setminus \{1_{\lambda}\}) \times (I\setminus \{1_I\})$ has either all entries the same, or $G^{P'}=\{a,b\}$ with $a$ appearing exactly once in each row and column.

\begin{lemma} Let $S=\mathcal{M}[G;I,\Lambda;P]$ be a normalised Rees matrix semigroup such that $\Gamma(S)$ is finitely coloured and not of generic type. Then $S$ is homogeneous if and only if $G$ is homogeneous with characteristic subgroup $G^P$ such that either  
\begin{enumerate}
\item $G^P=\{\epsilon\}$, so that $S$ is orthodox. 
\item $|I|=|\Lambda|=2$ with $G^P=\{\epsilon,a\}\cong \mathbb{Z}_2$ and $P'=(a)$. 
\item $|I|=|\Lambda|=3$ with $G^P=\{\epsilon,a,a^{-1}\}\cong \mathbb{Z}_3$ and $P'$ is of the form
\[ \left( \begin{array}{cc}
a & a^{-1} \\
a^{-1} & a 
\end{array} \right).\]  
\item $|I|=|\Lambda|=4$ with $G^P=\{\epsilon,a\}\cong \mathbb{Z}_2$ and $P'$ is of the form   
\[ \left( \begin{array}{ccc}
\epsilon & a & a \\
a & \epsilon & a \\
a & a & \epsilon
\end{array} \right).\] 
\end{enumerate}
\end{lemma}   

\begin{proof} Suppose $S$ is homogeneous. Then $G$ is homogeneous with characteristic subgroup $G^P$ by Propositions \ref{prop:G, E hom} and \ref{GP char}. Since $\Gamma(S)$ is finitely coloured, it is thus homogeneous by Proposition \ref{prop: graph hom}. Hence by Theorem \ref{thm:edge class} $P'$ has either all entries the same, or $G^{P'}=\{a,b\}$ with $a$ appearing exactly once in each row and column, so that $P'$ has the same number of rows and columns. Since the homogeneity of $S$ passes to $\langle E(S) \rangle = [G^P;I,\Lambda;P]$, we may assume that $S=\langle E(S) \rangle$ to show that $P'$ reduces to one of the four forms. 

Suppose first that $G^{P'}=\{a\}$.  If $a=\epsilon$ then $S$ is orthodox and we obtain case (1), so assume instead that $a\neq \epsilon$, so that $G^P=\{\epsilon,a\}\cong \mathbb{Z}_2$.  
 Suppose, seeking a contradiction, that $|I|>2$, and fix distinct $1_I,i,j\in I$.  Then $P$ contains the $\{1_\Lambda,\lambda\} \times \{1_I,i,j\}$ submatrix
\[ \left( \begin{array}{ccc}
\epsilon & \epsilon & \epsilon \\
\epsilon & a & a 
\end{array} \right).\] 
Consider a pair of f.g. Rees subsemigroups of $S$ given by 
\[ S_1=\mathcal{M}[G^P; \{1_I,i,j\},\{1_\Lambda\};P_1], \quad S_2=\mathcal{M}[G^P; \{1_I,i,j\},\{\lambda\};P_2]. 
\] 
so that $P_1=(\epsilon \, \epsilon \, \epsilon)$ and $P_2=(\epsilon \, a \, a)$. 
Let $\psi:\Gamma_{P_1}\rightarrow \Gamma_{P_2}$ be the isomorphism given by $1_I\psi=i$, $i\psi=j$, $j\psi=1_I$ and $1_\Lambda \psi=\lambda$. 
Let $\phi=[\text{Id}_{G^P},\psi,u_i,v_{\lambda}]$ be an isomorphism from $S_1$ to $S_2$, so that
\begin{equation} \label{eq:case 1}  \epsilon=v_{1_\Lambda} au_{1_I} = v_{1_\Lambda} a u_i = v_{1_\Lambda}  u_j,
\end{equation} 
which may be satisfied by $v_{1_{\Lambda}}=\epsilon=u_j$ and $u_{1_I}=a=u_i$, say. Extend $\phi$ to an automorphism $\phi'=[\text{Id}_{G^P},\psi',u_i',v_{\lambda}']$ of $S$. Then 
\[ \epsilon=v_{\lambda}' p_{\lambda\psi',i}u_{1_I}', \quad a = v_{\lambda}' p_{\lambda\psi',j}u_i', \quad a=v_{\lambda}' p_{\lambda\psi',1_I} u_j' =v_{\lambda}'u_j',
\] 
and $u_j'=au_{1_I}'=au_i'$ by \eqref{eq:case 1}. Hence, as $G^P$ is abelian with $a^2=\epsilon$, we have  
\[ \epsilon= v_{\lambda}'p_{\lambda\psi',i} (au_j') = (v_{\lambda}'u_j')p_{\lambda\psi',i} a = ap_{\lambda\psi',i} a = p_{\lambda\psi',i}, 
\] 
so that $\lambda\psi'=1_{\Lambda}$. Similarly, 
\[ a= v_{\lambda}' p_{\lambda\psi',j}(a u_j')= p_{\lambda\psi',j} =p_{1_\Lambda ,j}=\epsilon, 
\] 
and we arrive at our desired contradiction. Thus $|I|=2=|\Lambda|$, and case (2) is obtained. 

Now suppose $G^{P'}=\{a,b\}$, and suppose the edges of $\Gamma(S)$ coloured by either $a$ or $b$ forms a perfect matching, so that $|I|=|\Lambda|\geq 3$ and $G^P$ is isomorphic to either $\mathbb{Z}_2$ or $\mathbb{Z}_3$, and thus abelian. 

Suppose first that $b=a^{-1}$, so that $G\cong \mathbb{Z}_3$, and suppose without loss of generality that $a$ appears exactly once in each row and column of $P$. Consider the subsemigroup $T=\{(1_I,a^n,1_\Lambda ):n=-1,0,1\}$ of $S$. Then we may extend the unique non-identity automorphism of $T$ to an automorphism $\phi=[\theta,\psi,\epsilon,\epsilon]$ of $S$, noting the use of Corollary \ref{cor: inner}. Then 
\[ (1_I,a,1_\Lambda)\phi = (1_I,a\theta,1_{\Lambda}) = (1_I,a^{-1},1_{\Lambda}),
\] 
so that $a\theta=a^{-1}$. Suppose, seeking a contradiction, that $|I|>3$, so that there exist $i,j,k\in I\setminus \{1_I\}$. Since row $\lambda$ contains  $a$ exactly once, we may assume without loss of generality that  $p_{\lambda,i}=p_{\lambda,j}=a^{-1}$, so that 
\[  p_{\lambda\psi,i\psi} = p_{\lambda,i}\theta =   a =  p_{\lambda,j}\theta 
\] 
and similarly $p_{\lambda\psi,j\psi}=a$, contradicting $ \lambda\psi$ containing $a$ only once. Hence $|I|=|\Lambda|=3$, and case (3) is achieved. 

Suppose instead that $b\neq a^{-1}$, so that $b=\epsilon$ as $G^P=\{\epsilon,a,b\}$ forms a group. Then $P$ contains a $\{1_\Lambda,\lambda,\mu\} \times \{1_I,i,j\}$ submatrix given by 
\[ \left( \begin{array}{ccc}
\epsilon & \epsilon & \epsilon \\
\epsilon & a & \epsilon \\
\epsilon & \epsilon & a
\end{array} \right).\] 
We  study the Rees subsemigroups of the form 
\[ S_1=\mathcal{M}[G^P; \{1_I,i,j\},\{1_\Lambda\};(\epsilon \, \epsilon \, \epsilon)], \quad S_2=\mathcal{M}[G^P; \{1_I,i,j\},\{\lambda\};(\epsilon \, a \, \epsilon)],
\] 
and let $\phi=[\text{Id}_{G^P},\psi,u_i,v_\lambda ]$ be the isomorphism from $S_1$ to $S_2$, where $\psi$ fixes $1_I$, and swaps $i$ and $j$, and $1_\Lambda \psi=\lambda$. For $\phi$ to be a morphism we require 
\begin{equation} \label{eq:8} \epsilon=v_{1_\Lambda} u_{1_I} = v_{1_\Lambda} u_i = v_{1_\Lambda} a u_j, 
\end{equation} 
which is satisfied by $v_{1_\Lambda}=\epsilon=u_{1_{I}}=u_i$ and $u_j=a$, say. Extend $\phi$ to $\phi'=[\text{Id}_{G^P},\psi',u_i',v_{\lambda}']\in \text{Aut}(S)$, noting that as $1_{I}$ is fixed by $\psi$ it follows by the proof of Corollary \ref{cor: inner} that $v_{\sigma}'=v_{1_\Lambda}'$ for each $\sigma\in \Lambda$. By \eqref{eq:8}, we have  
\begin{align*}  &   p_{\sigma,i}=v_\sigma ' p_{\sigma\psi,j}u_i'= v_{1_{\Lambda}}' p_{\sigma\psi,j} u_i' = p_{\sigma\psi,j}, \\
& p_{\sigma,j}=v_{\sigma}' p_{\sigma\psi,i} u_j' = v_{1_{\Lambda}}' p_{\sigma\psi,i} u_j' = ap_{\sigma\psi,i}.
\end{align*}
 In particular, $a=p_{\lambda,i}=p_{\lambda\psi,j}$ and $\epsilon=p_{\lambda,j}=ap_{\lambda\psi,i}$, so that $p_{\lambda\psi,i}=a$. Since $a\neq \epsilon$, it follows that $\lambda\psi\notin \{1_{\Lambda},\lambda,\mu\}$, so  $|I|=|\Lambda|\geq 4$, with each row and column of $P'$ containing $\epsilon$ exactly once. Let $\gamma \in \Lambda\setminus \{1_{\Lambda},\lambda,\mu\}$. Then $p_{\gamma,j}=a$ and $p_{\gamma,j}=ap_{\gamma\psi,i}$, so that $p_{\gamma\psi,i}=\epsilon$. Hence $\gamma\psi\in \{1_{\Lambda},\mu\}$. However, $p_{\gamma,i}=a$ and $p_{\gamma,i}=p_{\gamma\psi,j}$, so that $\gamma\psi=\mu$. Hence $|I|=|\Lambda|=4$, and we arrive at case (4).

 Conversely, case (1) is homogeneous by Theorem \ref{orth}. For each of the cases (2), (3) and (4), the subsemigroup $\langle E(S) \rangle =\mathcal{M}[G^P;I,\Lambda;P]$  can be verified to be homogeneous by using the Semigroups package \cite{GAP} for the computational algebra system GAP \cite{Gap2}; see \cite{Comp} for details of the computation.  Hence $S$ is homogeneous by Theorem \ref{thm: E,G}. 
\end{proof}

Given that a  classification of homogeneous finite groups is known, we thus obtain  a  classification of homogeneous finite completely simple semigroups. However, using the theorem above in practice requires the understanding of which homogeneous groups possess a characteristic subgroup isomorphic to $\mathbb{Z}_2$ or $\mathbb{Z}_3$. By a simple application of homogeneity, this is equivalent to the homogeneous group possessing a unique copy of $\mathbb{Z}_2$ or $\mathbb{Z}_3$. 

For example, the quaternions and the special linear groups $SL_2(5)$ and $SL_2(7)$ possess a unique copy of $\mathbb{Z}_2$, but not a unique copy of $\mathbb{Z}_3$. On the other hand,  the linear groups $L_2(5)$ and $L_2(7)$ do not fall into either category. 

Moreover, it is clear from the work of Cherlin and Felgner \cite{Cherlin91} that a homogeneous abelian $p$-group with a unique copy of $\mathbb{Z}_p$ for some prime $p$ is isomorphic to either $\mathbb{Z}_p$ or the Pr{\"u}fer $p$-group $\mathbb{Z}[p^{\infty}]$. From this, all homogeneous abelian groups with a unique copy of $\mathbb{Z}_p$ may be easily built. 

%This result extends to the solvable case by also allowing the semidirect product of an abelian group with a unique pair of elements of order 3 with $\mathbb{Z}_2$. 

%
%For example, a finite group is homogeneous with a unique involution if and only if it is isomorphic to $U \times V$ where $(|U|,|V|)=1$,  $U$ \text{ is abelian of odd order with all Sylow subgroups homocyclic } and $V$ is one of
% \begin{align*}
% & \mathrm{i)} \, \mathbb{Z}_2, \\ 
% & \mathrm{ii) }\, W \rtimes  \mathbb{Z}_{2^n} (n \geq 2) \text{ where  W is Abelian of odd order with all Sylow subgroups homocylic,  } \\
%\,  & \text{and } \mathbb{Z}_{2^n} \text{ inverses all elements of } W, \\ 
% & \mathrm{iii) } \,  Q_8 \text{ (the quaternion group),}\\
% & \mathrm{iv) } \, Q_8 \rtimes \mathbb{Z}_3, \\
% & \mathrm{v) } \, \mathbb{Z}_3^2 \rtimes Q_8, \\
% & SL_2(5) \text{ and } SL_2(7),   
% \end{align*}
% 
% \begin{enumerate}
% \item $(W\times \mathbb{Z}_3)\rtimes \mathbb{Z}_{2}$, where $W$ is Abelian of odd order with all Sylow subgroups homocyclic and no elements of order 3, 
%\item $\mathcal{Q}_8 \rtimes \mathbb{Z}_3$, 
%\item $\mathcal{Q}(64) \rtimes \mathbb{Z}_3$. 
% \end{enumerate}

 \subsection{The generic case }
 
 In this section we consider the homogeneity of the final two cases: where $\Gamma(S)$ is  of generic type or $G^P$ is infinite. In either case we have that both $I$ and $\Lambda$ are infinite, and by the following result we need only consider the generic case: 
 
 \begin{lemma}\label{lemma:generic} Let $S=\mathcal{M}[G;I,\Lambda;P]$ be a homogeneous non-orthodox Rees matrix semigroup with $G^P$ infinite. Then $\Gamma(S)$ is $G^P$-generic. 
 \end{lemma} 
 
 \begin{proof}   We claim that any $x\in G^{P'}$ is repeated infinitely many times in $P'$. 
 Let $x=p_{\lambda,i}\in C(i)$ for some $i\neq 1_I$. For some fixed $n>1$, let $x, p_{\lambda,i_1},\dots, p_{\lambda,i_n}$ be distinct non-identity elements of $R(\lambda)$. 
Consider the f.g. subgroup $H=\langle x,p_{\lambda,i_k}: 1\leq  k \leq n\rangle$ of $G^P$, so that $T=\mathcal{M}[H;\{1_I, i, i_1,\dots,i_n\},\{1_\Lambda\};Q]$ is a f.g. Rees subsemigroup of $S$. 
  For each $1\leq k \leq n$, let $\psi_k$ be the automorphism of $\Gamma_Q$ which swaps $i_1$ and $i_k$, and fixes all other elements.
   Then $\phi_k=[\text{Id}_H,\psi_k,\epsilon,\epsilon]$ is an automorphism of $T$ by Corollary \ref{iso orth}, and so by the homogeneity of $S$ and Corollary \ref{cor: inner} we may extend $\phi_k$ to an automorphism $\phi'_k=[\theta_k,\psi'_k,\epsilon,\epsilon]$ of $S$.
   For each $h\in H$, the element  $(1_I,h,1_\Lambda)$ is fixed by $\phi_k'$, and so by Corollary \ref{cor: inner}  we have $p_{\lambda,i}=p_{\lambda\psi_k',i} = x$ and  $p_{\lambda,i_1}=p_{\lambda\psi_k', i_1\psi_k'}= p_{\lambda\psi_k',i_k}$. By considering each $1 < k \leq n$, it follows from the fact that $p_{\lambda,i_1},\dots,p_{\lambda,i_n}$ are distinct that there are $n$ distinct elements $\lambda, \lambda\psi_2',\dots,\lambda\psi_n'$ of $\Lambda$. Hence $C(i)$ contains $n$ copies of $x$, for arbitrarily large $n$, and the claim follows.

We now claim that $G^P=G^{P'}$, for which it suffices to show that $\epsilon$ appears in $P'$. Fix some $\lambda\in \Lambda'$,  and let $a\in R(\lambda)$ with $a\neq \epsilon$. 
 Then $T=[\langle a \rangle;\{1_I\},\{1_\Lambda , \lambda\}; Q]$ is a f.g. Rees subsemigroup of $S$.  
 Let $\psi$ be the automorphism of $\Gamma_Q$ which swaps $1_\Lambda$ and $\lambda$, so that $[\text{Id}_{\langle a \rangle},\psi,\epsilon,\epsilon]$ is an automorphism of $T$ by Corollary \ref{iso orth}.  Extend the isomorphism to an automorphism $[\theta,\psi',u_i,v_{\lambda}]$ of $S$, noting that as $1_I$ is fixed, we have that $v_{\mu}=u_{1_I}^{-1}$ for all $\mu\in \Lambda$. By the previous claim, there exist infinitely many $i\in I$ such that $p_{\lambda,i}=a$. Fix $i\in I$ such that $p_{\lambda,i}=a$ and $i\psi'\neq 1_I$. 
 Then as $\lambda\psi'=1_\Lambda$ we have 
 \[ p_{\lambda,i}\theta=a\theta = u_{1_I}^{-1} p_{\lambda\psi',i\psi'} u_i = u_{1_I}^{-1} u_i
 \]
 and so for any $\gamma\neq \lambda$ such that $p_{\gamma,i}=a$ we have 
 \[ p_{\gamma,i}\theta = a\theta = u_{1_I}^{-1} p_{\gamma\psi',i\psi'} u_i \Rightarrow p_{\gamma\psi',i\psi'} = \epsilon. 
 \] 
Since both $i\psi'\neq 1_I$ and $\gamma\psi'\neq 1_\Lambda$ we have that $\epsilon\in G^{P'}$, thus proving the claim.

We now show that $\Gamma(S)$ is $G^P$-generic. Let $J=\{j_1,\dots,j_r\}$ be a finite subset of $I'$, and $\alpha:J\rightarrow G^P$ a map given by $i_t\alpha=x_t$. Then by the first claim there exists $\mu\in \Lambda$ and $k_1,\dots,k_r\in I$  such that $p_{\mu,k_t}=x_t$ for each $t$. Let $T=\mathcal{M}[\langle x_1,\dots,x_r \rangle;J\cup \{1_I, k_1,\dots,k_r\},\{1_{\Lambda}\};Q]$ be a Rees subsemigroup of $S$, noting that $Q$ contains only the identity element. Let $\psi$ be the automorphism of $\Gamma_Q$ which swaps $j_t$ and $k_t$ for each $1\leq t \leq r$, and fixes $1_I$ and $1_{\Lambda}$. Then by Corollary \ref{iso orth}, [Id$_{\langle x_1,\dots,x_r \rangle},\psi,\epsilon,\epsilon]$ is an automorphism of $T$, which we may extend to an automorphism $[\theta,\psi',\epsilon,\epsilon]$ of $S$ by Corollary \ref{cor: inner}. Then $(1_I,x_t,1_{\Lambda})$ is fixed, so that
\[ p_{\mu,k_t}\theta = p_{\mu\psi',j_t} = p_{\mu,k_t} = j_t \alpha. 
\] 
Hence $\mu\psi'\in \Lambda'$ is a witness for $J$. A dual argument holds for finite subsets of $\Lambda'$, and so $\Gamma(S)$ is $G^P$-generic.  
 \end{proof}
 
Given a group $G$ and characteristic subgroup $H$, we let $\mathcal{CS}(G;H)$ denote the class of all f.g. completely simple semigroup which are isomorphic to a normalised Rees matrix semigroup of the form $\mathcal{M}[K;J,M;Q]$ with $K\in \text{age}(G)$ and $\langle K^Q  \rangle \in \text{age}(H)$. 
 
 \begin{lemma} Let $S=\mathcal{M}[G;I,\Lambda;P]$ be homogeneous with $\Gamma(S)$ of generic type. Then age($S$)= $\mathcal{CS}(G;G^P)$.   
 \end{lemma}

 \begin{proof} We claim that if $Q$ is a finite matrix over $G^P$, then $Q$ appears as a submatrix of $P$. We proceed by induction on the number of rows of $Q$, noting that the base case is immediate from the previous lemma. For some $m\in \mathbb{N}$, assume the claim holds for all  matrices over $G^P$ with less than $m$ rows. Let $Q=(q_{k,\ell})_{1\leq \ell \leq n, 1\leq k \leq m}$ for some $n\in \mathbb{N}$. Then by the inductive hypothesis, the submatrix of $Q$ obtained by removing row $m$ appears as a submatrix $P^*$ of $P$, say $q_{k,\ell}=p_{\lambda_k,i_{\ell}}$. 
Since $\Gamma(S)$ is $G^P$-generic there exist infinity many $\lambda\in \Lambda'$ such that $p_{\lambda,i_{\ell}}=q_{m,\ell}$ for each $1\leq \ell \leq n$. The claim then follows by choosing $\lambda$ such that $\lambda \neq \lambda_k$ for each $1 \leq k \leq m$. 
 
% By the base case, there exists $\lambda \in \Lambda$ and $j_1,\dots,j_n\in I$ such that $q_{m,k} = p_{\lambda,j_k}$. Since $\Gamma(S)$ is $G(P)$-generic by the previous result there exists $\lambda'\in \Lambda$ such that $p_{\lambda',i_k}= q_{m,k}$. In fact we may use our usual technique to show that there exists infinitely many $\lambda'$ with this property. Indeed, let $p_{\lambda',j}=a$ and $p_{\lambda',k_1}=\cdots p_{\lambda',k_r}=b\neq a$. Then by fixing each $i_k$, and by swapping $j$ with either $k_r$ we obtain from Lemma 6.6 that there exists $r$ such $\lambda'$, for any $r>1$. We may thus assume that $\lambda'$ is not a row in $P^*$, and so $P^*$ with row $\lambda'$ added gives the matrix $Q$. 
% 
 Now let $T$ be a member of $\mathcal{CS}(G;G^P)$, so we may assume without loss of generality that $T=\mathcal{M}[K;J,M;Q]$, where $T$ is normalised, $K$ is a f.g. subgroup of $G$, and $K^Q$ is a subset of $G^P$. By the previous claim, $Q$ forms a submatrix of $P$, and so $T$ forms a Rees subsemigroup of $S$. Hence $\mathcal{CS}(G;G^P)$ is a subclass of age($S$). The converse is immediate. 
 \end{proof}

 \begin{proposition}  Let $G$ be a homogeneous group with characteristic subgroup $H$. Then $\mathcal{CS}(G;H)$ forms a Fra\"iss\'e class. 
 \end{proposition} 
 
 \begin{proof} Note that $H$, being a characteristic subgroup of $G$, is homogeneous, and so age($G$) and age($H$) form Fra\"iss\'e classes. By construction $\mathcal{K}=\mathcal{CS}(G;H)$ is closed under isomorphism. If $K$ is a f.g. group, $A$ is a finite subset of $A$ and $I$ and $\Lambda$ are finite index sets, then the number of Rees matrix semigroups $\mathcal{M}[K;I,\Lambda;Q]$ such that $Q^K=A$ is finite. Hence, as  age($G$) is countable, it follows that $\mathcal{CS}(G;H)$ is countable. Similarly, the hereditary property is inherited from  age($G$) and age($H$). We now show that $\mathcal{K}$ has the AP, from which the proof can be easily adapted to show the JEP. 
 
The proof of the AP follows closely to the argument given by Clarke in  \cite{Clarke} to show that the variety of completely simple semigroups whose subgroups lies in some variety of groups has the AP. Let $[M_0;M_1,M_2]$ be an amalgam in age($S$), where $M_k=\mathcal{M}[G_k;I_k,\Lambda_k;P_k]$ $(k=0,1,2)$. We may again assume that $H_1\cap H_2=H_0$, $I_1\cap I_2 = I_0$, $\Lambda_1\cap \Lambda_2=\Lambda_0$, and each $M_k$ is normalised with $1_{I_1}=1_{I_0}=1_{I_2}$ and $1_{\Lambda_1}=1_{\Lambda_0}=1_{\Lambda_2}$. We may also assume that $p_{\mu,j}^{(0)}=  p_{\mu,j}^{(1)} =p_{\mu,j}^{(2)}$ for each $j\in I_0$,  $\mu \in \Lambda_0$. 

The amalgam $[G_0;G_1,G_2]$ can be be embedded in a group $K\in \text{age}(G)$, by embeddings $\phi_1$ and $\phi_2$ of $G_1$ and $G_2$, respectively. Let $\bar{I}=I_1\cup I_2$ and $\bar{\Lambda}=\Lambda_1\cup \Lambda_2$.  Define the $\bar{\Lambda} \times \bar{I}$ matrix $Q=(q_{\lambda,i})$ by 
 \[ q_{\lambda,i} = \begin{cases} p_{\lambda,i}^{(1)}\phi_1, & \mbox{if } \lambda \in \Lambda_1 \mbox{ and } i\in I_1, \\  p_{\lambda,i}^{(2)}\phi_2, & \mbox{if } \lambda \in \Lambda_2 \mbox{ and } i\in I_2, \\
\epsilon & \mbox{otherwise}, \end{cases}
\] 
 and put $T=\mathcal{M}[K;\bar{I},\bar{\Lambda};Q]$. Note that $T$ is normalised along row $1_{I_0}$ and down column $1_{\Lambda_0}$. The map $\theta_k = [\phi_i,\iota_k,\epsilon,\epsilon]$ from $M_k$ to $T$ ($k=1,2$) is an embedding by Theorem \ref{iso thm}, where $\iota_k$ is the inclusion embedding. Moreover, $\theta_1$ and $\theta_2$ clearly agree on $M_0$, so $[M_0;M_1,M_2]$ can be embedded in $T$. It therefore suffices to show that $T$ is a member of $\mathcal{CS}(G;H)$. 
% as $p_{\lambda,1_{I_0}}\phi_k^{(k)} = e = p_{i,1_{\Lambda_0}}\phi_k^{(k)}$ for each $i\in I_k, \lambda\in \Lambda_k$ ($k=1,2$).

Note that if $\theta:K_1\rightarrow K_2$ is an embedding of members of age($G$), and if $K_1'\leq K_1$ is a member of age($H$), then so too is $K_1'\theta$ (simply extend the isomorphism between $K_1'$ and $K_1'\theta$ to an automorphism of $G$, noting that $H$ is a characteristic subgroup of $G$). Hence, as $\langle G_1^{P_1} \rangle$ and $\langle G_2^{P_2} \rangle$ are members of $\text{age}(H)$, so too are $\langle G_1^{P_1}\phi_1\rangle$ and $\langle G_2^{P_2}\phi_2 \rangle$, and thus $\langle K^Q \rangle$, being generated by these groups, is a member of age($H$). 
 \end{proof}
 
Every homogeneous Rees matrix semigroup of generic type can therefore be built from a group $G$, a characteristic subgroup $H$, and an $H$-generic bipartite graph.  As a consequence we obtain all homogeneous Rees matrix semigroups with infinite sandwich matrix, as either a direct product of a group and a rectangular band, or the Fra\"iss\'e limit of some $\mathcal{CS}(G;H)$. We summarise our findings: 
 
 \begin{theorem} \label{thm: classify} A completely simple semigroup $S$ is homogeneous if and only if there exists a homogeneous group $G$ with characteristic subgroup $H$ such that 
 \begin{enumerate}
 \item $S=G\times B$ for some rectangular band $B$;
 \item $S=\mathcal{M}[G;\underline{2},\underline{2};P_1]$ where $H=\{\epsilon,a\}\cong\mathbb{Z}_2$ and $P'_1=(a)$; 
 \item $S=\mathcal{M}[G;\underline{3},\underline{3};P_2]$ where $H=\{\epsilon,a,a^{-1}\}\cong\mathbb{Z}_3$ and $\Gamma(S)$ is $\{a,a^{-1}\}$-edge coloured, with edges having $a$ as colour forming a perfect matching;  
 \item $S=\mathcal{M}[G;\underline{4},\underline{4};P_3]$ where $H=\{\epsilon,a\}\cong\mathbb{Z}_2$ and $\Gamma(S)$  is $\{\epsilon,a\}$-edge coloured, with edges having $\epsilon$ as colour forming a perfect matching;  
 \item $S$ is the Fra\"iss\'e limit of $\mathcal{CS}(G;H)$, that is, the Rees matrix semigroup $\mathcal{M}[G;I,\Lambda;P]$ with $G^P=H$ and $\Gamma(S)$ being $H$-generic. 
 \end{enumerate}
 \end{theorem}

 \section{Homogeneous semigroups} 
 
A semigroup $S$ is \textit{inverse} if every element has a unique inverse, that is, for each $x\in S$ there exists a unique $y\in S$ with $x=xyx$ and $y=yxy$. Inverse semigroups may be naturally considered as a unary semigroup, with unary operation mapping elements to their inverses.  In \cite{Quinninv} the author showed that the condition that an inverse semigroup is homogeneous as a unary semigroup is stronger than the condition that it is homogeneous as a semigroup, and simple examples were  constructed to show that the two concepts of homogeneity differ. In this section we continue this line of work by investigating the homogeneity of  completely simple semigroups as semigroups. The key difference is that we shall be considering isomorphisms between \textit{all} subsemigroups, and as such it will be a stronger, albeit less natural, condition. This work is further motivated by Proposition 6.2 and Lemma 6.3, which state that regular homogeneous semigroups with  either elements of infinite order or finitely many idempotents are completely simple.  Given a completely simple semigroup with a subset $X$, to avoid notation clashes we shall let $\langle X \rangle_S$ denote the subsemigroup generated by $X$ (rather than the unary subsemigroup which we denoted by $\langle X \rangle$). 

A semigroup $S$ is called \textit{periodic} if every element is of finite order, that is, if the monogenic subsemigroup $\langle x \rangle_S$ is finite for each $x\in S$. On the class of periodic completely simple semigroups, every subsemigroup is necessarily completely simple (folklore, and remarked upon in \cite{Ant}). Hence our two notions of homogeneity for a completely simple semigroup intersect in this case: 

\begin{lemma}\label{periodic same} A periodic completely simple semigroup is a homogeneous semigroup if and only if it is a homogeneous completely simple semigroup. 
\end{lemma} 

On the other hand, subsemigroups of non-periodic completely simple semigroups can be unwieldy, and are considered in \cite{Ant}. 
  Indeed, even  in the case of groups, it is not known if a homogeneous group is necessarily homogeneous as a semigroup (the abelian case is proved to hold in \cite{Quinninv}). \newline 
 
 \noindent {\bf Open Problem 1:} Characterise which homogeneous completely simple semigroups are homogeneous semigroups. \newline

In Corollary 6.3 of \cite{Quinninv}, a regular homogeneous semigroup with a non-periodic element contained in a subgroup is shown to be completely simple. We now generalize this to show that an answer to Open Problem 1 would in fact classify all non-periodic regular homogeneous semigroups. 

Each regular semigroup $S$ comes equipped with a quasi-order $\leq_{\mathcal{R}}$, known as \textit{Green's right quasi-order}, defined by $a  \leq_{\mathcal{R}}  b$ if and only if there exists $u\in S$ such that $a=bu$. Recall that the associated equivalence relation is Green's $\mathcal{R}$-relation. Note that $\leq_{\mathcal{R}}$ is preserved by morphisms, that is, if $\phi:S\rightarrow T$ is a morphism of semigroups and if $a \leq_{\mathcal{R}} b$ in $S$ then $a\phi \leq_{\mathcal{R}} b\phi$ in $T$. 

 Recall that the set of idempotents $E(S)$ of a semigroup comes equipped with a \textit{natural order} $\leq$, defined by $e\leq f$ if and only if $ef=fe=e$.
 Any $e\in E(S)$ is a left identity for its $\mathcal{R}$-class. Consequently, if $e,f\in E(S)$ then $e \leq_{\mathcal{R}} f$ if and only if $ef=e$, so that $e\leq f$ implies that $e \leq_{\mathcal{R}} f$.

\begin{proposition}\label{non per cr}  Let $S$ be a regular homogeneous semigroup. If $S$ is non-periodic then $S$ is completely simple.  
\end{proposition} 

\begin{proof} By \cite[Theorem 3.3.3]{Howie94} it suffices to show that each idempotent of $S$ is primitive. Let $x$ be an element of $S$ of infinite order. If $x$ is contained in a subgroup of $S$ then $S$ is completely simple by \cite{Quinninv}, so we assume the contrary. Since $S$ is regular we may pick some idempotent $e$ such that $x \, \mathcal{R} \, e$. Consider the subsemigroup of $S$ given by $A=\langle x,xe \rangle_S$. Since $e$ is a left identity for $x$ we have for any $n,m\in \mathbb{N}$, 
\[ (x^ne)(x^me)=x^{n+m}e, \quad x^n(x^me)=x^{n+m}e, \quad (x^ne)x^m = x^{n+m},
\] 
and so $A=\{x^n,x^ne:n\in \mathbb{N}\}$. Notice that $xe$ has infinite order, since if $x^ne=x^me$ then $x^{n+1}=(x^ne)x=(x^me)x=x^{m+1}$, a contradiction. It follows from the multiplication in $A$ that the map swapping $x^n$ with $x^ne$ ($n\in \mathbb{N}$) is an automorphism of $A$. 
By the homogeneity of $S$ we may extend the map to an automorphism $\theta$ of $S$. Then $(xe)\theta=xe (e\theta)=x$, and so $x \, \mathcal{R} \, xe$. Since $e$ is an identity of $xe$ we have, for any idempotent $f\geq e$,
\[ f(xe)=f(exe)=(fe)xe=e(xe)=xe = xef. 
\] 
Hence every idempotent $f\geq e$ is an identity of $xe$, and so the map $\phi$ from $\langle xe,e \rangle_S$ to $\langle xe, f \rangle_S$  mapping $e$ to $f$ and fixing all other elements is an isomorphism. Extending $\phi$ to an automorphism of $S$, then  $xe \, \mathcal{R} \, f$, so that $e \, \mathcal{R} \, f$. Hence $e=f$, and so $e$ is a maximal idempotent under the natural ordering. Now let $g\in E(S)$. Since idempotents generate trivial semigroups, it follows by the homogeneity of $S$ that there exists an automorphism $\theta$ of $S$ such that $e\theta =g$, and so $g$ is also maximal. Hence all idempotents of $S$ are primitive, and so $S$ is completely simple. 
\end{proof}

%Note: the proof above does not work in the periodic case. If $x^n=x^{n+r}$ and $e=xx^{-1}$ then $xe$ has order $n-1$ and period $r$. Indeed, $x^{n-1}e = x^{n-1}xx^{-1}=x^{n+r}x^{-1} = x^{n+r-1}e$. 

\begin{lemma} A regular homogeneous semigroup with finite set of idempotents is a homogeneous completely simple semigroup. 
\end{lemma} 

\begin{proof}  Let $e\in E=E(S)$. Since $E$ is finite, there exists a primitive idempotent $f$ under the partial order $\leq$ on $E$. Then by the homogeneity of $S$  there exists an automorphism $\theta$ of $S$ such that $e\theta =f$. If there exists $g\in E$ such that $g\leq f$, then $g\theta\leq e \theta=f$, and so $g=f$ as $f$ is primitive. Hence all idempotents are primitive, and so $S$ is completely simple. The result then follows from Lemma \ref{periodic same}. 
\end{proof} 
 
 A consequence of the lemma above together with Theorem \ref{thm: classify} is that we now have a full classification of all finite regular homogeneous semigroups (where in Theorem \ref{thm: classify} the group $G$ is forced to be finite, and case (5) cannot hold).

 Note that if we drop the condition that $S$ is regular then the lemma no longer holds. Indeed, it is a simple exercise  to check that the monogenic semigroup $\langle a: a^4=a^2 \rangle_S$ is homogeneous, but not completely simple.

\section{Acknowledgements} 

The author would like to thank  James Mitchell for creating the vital GAP code with astonishing speed. Thanks also to Brennen Fagan  for tirelessly helped me  with the programming, and to Victoria Gould for her helpful comments. 
%I thank Victoria Gould for carefully reading the paper
%and correcting mistakes.

\end{document}